\documentclass[12pt]{amsart}
\usepackage{graphicx,amsmath,mathtools}
\usepackage{amssymb,amsthm}
\usepackage{epsfig}
\usepackage{epstopdf}
\usepackage{color}
\newtheorem{theorem}{Theorem}[section]
\newtheorem{lemma}[theorem]{Lemma}
\theoremstyle{definition}
\newtheorem{definition}[theorem]{Definition}
\newtheorem{example}[theorem]{Example}
\newtheorem{proposition}[theorem]{Proposition}
\theoremstyle{remark}
\newtheorem{remark}[theorem]{Remark}
\theoremstyle{conjecture}
\newtheorem{conjecture}[theorem]{Conjecture}
\theoremstyle{corollary}
\newtheorem{corollary}[theorem]{Corollary}
\numberwithin{equation}{section}

\begin{document}

\title{Equivalence of two definitions of set-theoretic Yang-Baxter homology}

\author{J\'{o}zef H. Przytycki}
\address{Department of Mathematics, The George Washington University, Washington DC, USA and University of Gda\'nsk}
\email{przytyck@gwu.edu}
\thanks{The first author was supported in part by the Simons Collaboration Grant-316446, and Dean's Research Chair award.}

\author{Xiao Wang}
\address{Department of Mathematics, The George Washington University, Washington DC, USA}
\email{wangxiao@gwu.edu}

\date{November 3, 2016}

\begin{abstract}
In 2004, Carter, Elhamdadi and Saito defined a homology theory for set-theoretic Yang-Baxter operators(we will call it the ``algebraic" version in this article).  In 2012, Przytycki defined another homology theory for pre-Yang-Baxter operators which has a nice graphic visualization(we will call it the ``graphic" version in this article).  We show that they are equivalent.  The ``graphic" homology is also defined for pre-Yang-Baxter operators, and we give some examples of it's one-term and two-term homologies.  In the two-term case, we have found torsion in homology of Yang-Baxter operator that yields the Jones polynomial.
\end{abstract}

\maketitle
\section{Introduction}
\markboth{\hfil{\sc Two definitions of set-theoretic Yang-Baxter homology}\hfil}{J\'{o}zef H. Przytycki and Xiao Wang}
The Yang-Baxter equation was introduced independently by C.N. Yang(1967)\cite{Yan} and R.J. Baxter(1972)\cite{Bax}.  It is well known that a certain solution of  Yang-Baxter equation give rise to the Jones polynomial \cite{Jon-2}. In 2004, Carter, Elhamdadi and Saito defined a (co)homology theory for set-theoretic Yang-Baxter operators, from which they gave a way to generate link invariants, cocycle invariants\cite{CES}. In 2012, Przytycki gave a graphical definition of homology for a pre-Yang-Baxter operator \cite{Prz-2}.  We provide the definitions of two homology theories for set-theoretic Yang-Baxter operators in Section 2 and show their equivalence in Section 3.  In Section 4, we give definitions of one-term and two-term homology of pre-Yang-Baxter operators, and  show examples, in particular, we find torsion in two-term homology of Yang-Baxter operator that yields the Jones polynomial.
We start from basic definitions.
\begin{definition}\label{Definition 1.1}
Let $k$ be a commutative ring and $X$ be a set.  Consider $V=kX,$ the free $k-$module generated by $X$.  If a $k-$linear map, $R:$ $V\otimes V \to V\otimes V$, satisfies the following equation\\
$(R\otimes Id_{V})\circ (Id_{V}\otimes R)\circ (R\otimes Id_{V})=(Id_{V}\otimes R)\circ (R\otimes Id_{V})\circ (Id_{V}\otimes R),$\ \\
then we say $R$ is a pre-Yang-Baxter operator. The equation above is called a Yang-Baxter equation.  If, in addition, $R$ is invertible, then we say $R$ is a Yang-Baxter operator.  From now on, we assume $k$ is a commutative ring with identity whenever we deal with Yang-Baxter operators.
\end{definition}

\begin{definition}\label{Definition 1.2}
Let $X$ be a set.  If $R:$ $X\times X \to X\times X$ is a function that satisfies \\
$(R\times Id_{X})\circ (Id_{X}\times R)\circ (R\times Id_{X})=(Id_{X}\times R)\circ (R\times Id_{X})\circ (Id_{X}\times R),$\\
then we say $R$ is a pre-set-theoretic Yang-Baxter operator and the equation above is a set-theoretic Yang-Baxter equation. If, in addition, $R$ is invertible, then we say $R$ is a set-theoretic Yang-Baxter operator.
\end{definition}

Pre-set-theoretic Yang-Baxter operator leads to Yang-Baxter operator by putting $V=kX,$ and extending $R:$ $X\times X \to X\times X$ to $R:$ $V\otimes V \to V\otimes V$.  This Yang-Baxter operator is still called pre-set-theoretic Yang-Baxter operator.

\section{Set-theoretic Yang-Baxter homology theories}

Given a pre-set-theoretic Yang-Baxter operator $R,$ we have two approaches to homology built on $R$. The ``algebraic" version defined in \cite{CES} and the ``Graphic" version in \cite{Prz-2}.  We discuss them in the next two subsections.  We prove the equivalence of them in Section 3.

We first review the ``algebraic" version of set-theoretic Yang-Baxter homology theory based on\cite{CES} and then introduce the ``graphic" version of set-theoretic Yang-Baxter homology theory.

\subsection{``Algebraic" homology of Carter, Elhamdadi, and Saito}

\begin{definition}\label{Definition 2.1.1}
The set $X$ together with a set-theoretic Yang-Baxter operator $R$ , $(X,R), $ is called in \cite{CES} a Yang-Baxter set.  We represent a function $R$ by $R(x_{1},x_{2})=(R_{1}(x_{1},x_{2}),R_{2}(x_{1},x_{2})).$

We use the following notation. Let $\mathcal{I}_{n}$ be the $n-$dimensional cube $I^{n}$ ($I=[0,1]$) regarded as a CW (cubical) complex, where $n$ is a positive integer.\footnote{We deal here with a co-pre-cubic set ($X_{k},d^{i}_{\epsilon})$  where $X_{k}=I^{k}$ and co-face maps $d^{i,k-1}_{\epsilon}:X_{k-1} \to X_{_{k}} $ defined by $d^{i}_{\epsilon}(x_{1},x_{2},...,x_{k-1})=(x_{1},x_{2},...,x_{i-1},\epsilon,x_{i},...x_{k-1})$; they satisfy $d^{j}_{\delta}d^{i}_{\epsilon}=d^{i}_{\epsilon}d^{j-1}_{\delta}$ where $i<j.$} Denote the $k-$skeleton by $\mathcal{I}^{(k)}_{n}$ with orientation given by the order of coordinate axes. In particular, every $2-$face can be written as $ {\epsilon_{1}} \times \cdots \times {\epsilon_{i-1}} \times I_{i} \times {\epsilon_{i+1}} \times \cdots \times{\epsilon_{j-1}} \times I_{j}\times{\epsilon_{j+1}} \times \cdots \times{\epsilon_{n}},$ for some $1\leq i <j\leq n,$ where $\epsilon_{k}=0$ or $1,$ and $I_{i},$ $I_{j}$ denote two copies of $I$ at the $i$th, $j$th positions, respectively.
\end{definition}

\begin{definition}\label{Definition 2.1.2}
The Yang-Baxter coloring of $\mathcal{I}_{n}$ by a Yang-Baxter set $(X,R)$ is a map $L: E(\mathcal{I}_{n}) \to X$, where $ E(\mathcal{I}_{n})$ denotes the set of edges of   $\mathcal{I}_{n},$ with each edge oriented as above, such that if

 $L({\epsilon_{1}} \times \cdots \times \times I_{i} \times  \cdots \times 0_{j} \times \cdots \times{\epsilon_{n}})=x,$

$L({\epsilon_{1}} \times \cdots \times \times 1_{i} \times  \cdots \times I_{j} \times \cdots \times{\epsilon_{n}})=y,$ 

then

$L({\epsilon_{1}} \times \cdots \times \times 0_{i} \times  \cdots \times I_{j} \times \cdots \times{\epsilon_{n}})=R_{1}(x,y),$

$L({\epsilon_{1}} \times \cdots \times \times I_{i} \times  \cdots \times 1_{j} \times \cdots \times{\epsilon_{n}})=R_{2}(x,y),$\\ \

\centerline{\includegraphics[scale=0.5]{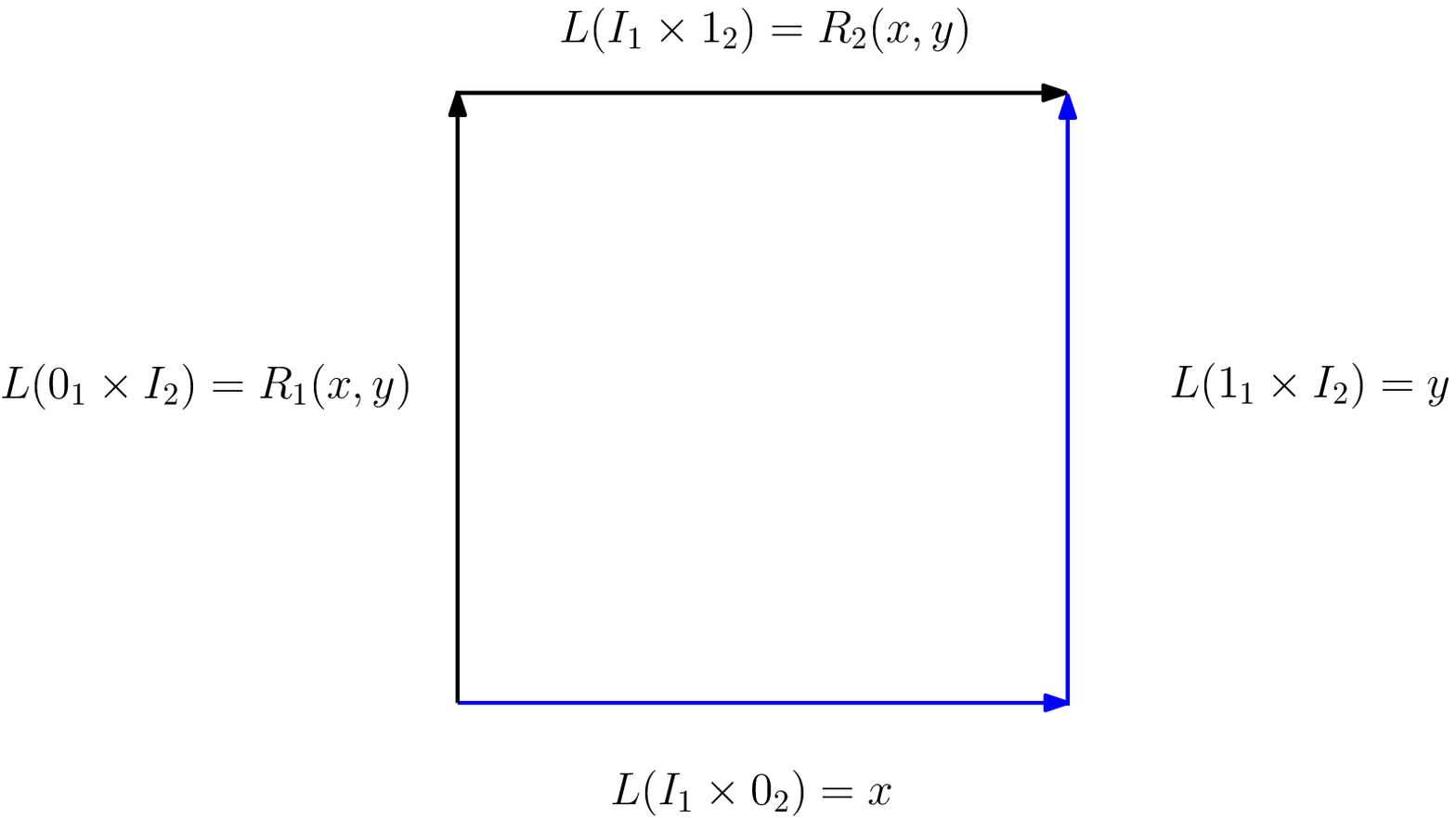}  }
\centerline{\footnotesize{Figure. 1.1; local behavior of Yang-Baxter coloring}}

\end{definition}\ \\

\begin{definition}\label{Definition 2.1.3}
The initial path in $\mathcal{I}_{n}$ is the sequence of edges of $\mathcal{I}_{n}$,$(e_{1},\cdots,e_{n}),$ where

$e_{1}=I_{1} \times 0_{2} \times \cdots \times 0_{n},$

$e_{2}=1_{1} \times I_{2} \times 0_{3} \times \cdots \times 0_{n},$

$\vdots$

$e_{n}=1_{1} \times 1_{2} \cdots \times 1_{n-1} \times I_{n},$

\end{definition}

\begin{lemma}\label{Lemma 2.1.4} (S.Carter, M. Elhamdadi, and M.Saito, 2004): 

Let $(X,R)$ be a Yang-Baxter set, and $(e_{1},\dots,e_{n})$ be the initial path of $\mathcal{I}_{n}$. For any n-tuple of elements of $X$, $(x_{1},\dots,x_{n}),$ there exists a unique Yang-Baxter coloring $L$ of $\mathcal{I}_{n}$ by $(X,R)$ such that $L(e_{i})=x_{i}$ for all $i=1,\dots,n.$
\end{lemma}

This lemma gives the following two properties.
\begin{enumerate}

\item Each edge has the color uniquely induced by the $n-$tuple asociated to the initial path of $\mathcal{I}_{n}.$

\item Each $k-$face $\mathcal{J}$ of $\mathcal{I}_{n}$ has its induced initial path determined by the order of coordinates. Therefore, we can associate to it the $k-$tuple $(y_{1},\dots,y_{k})$ determined by colors on its induced initial path. Denote this situation by $L(\mathcal{J})=(y_{1},\dots,y_{k})$

\end{enumerate}

From these two facts, we have a way to map an $n-$tuple to $(n-1)-$tuple through the face maps in cubic homology theory.

Recall that $\partial ^{C}_{n}$ denotes the $n-$dimensional boundary map in the cubical homology theory. Thus $\partial ^{C}_{n}(\mathcal{I}_{n})= \sum_{i=1}^{2n} \epsilon_{i}\mathcal{J}_{i},$ where $\mathcal{J}_{i}$ is an $(n-1)-$face and $\epsilon_{i}=\pm 1$ depending on whether the orientation of $\mathcal{J}_{i}$ matches the induced orientation.  For the induced orientation, we take the convention that the inward pointing normal to an $(n-1)-$face appears last in a sequence of vectors that specifies an orientation, and the orientation of the $(n-1)-$face is chosen so that this sequence agrees with the orientation of the $n-$cube.

Let $(X,R)$ be a Yang-Baxter set. Let $C^{YB}_{n}(X)$ be the free abelian group generated by $n-$tuples $(x_{1},\dots,x_{n})$ of elements of $X.$  Define a homomorphism $\partial^{A} _{n}:C^{YB}_{n}(X) \to C^{YB}_{n-1}(X)$ by $\partial^{A} _{n}((x_{1},\dots,x_{n}))=L(\partial^{C}_{n}(\mathcal{I}_{n}))=\sum_{i=1}^{2n} \epsilon_{i}L(\mathcal{J}_{i}).$ We have $\partial^{A}_{n-1} \circ \partial^{A}_{n}=0,$ and $(C^{YB}_{*}(X),\partial^{A}_{n})$  is a chain complex.  As usual, we can define $H^{A}_{n}=ker\partial^{A}_{n}/im\partial^{A}_{n+1}$ to be the ``algebraic" version of  Yang-Baxter homology group \cite{CES}.

\subsection{``Graphic" approach to Yang-Baxter homology}

In this homology theory, the chain groups are the same as before, that is $C^{YB}_{n}(X)=ZX^{n}.$  We define the boundary homomorphism $\partial^{G}_{n}:C^{YB}_{n}(X)\to C^{YB}_{n-1}(X)$ as follow,  $\partial^{G}_{n}=\sum_{i=1}^{n} (-1)^{i}d_{i,n},$ where $d_{i,n}=d^{l}_{i,n}-d^{r}_{i,n}$. 
 We can interpret the face maps through Figure 2.1\ \\

\centerline{\includegraphics[scale=0.8]{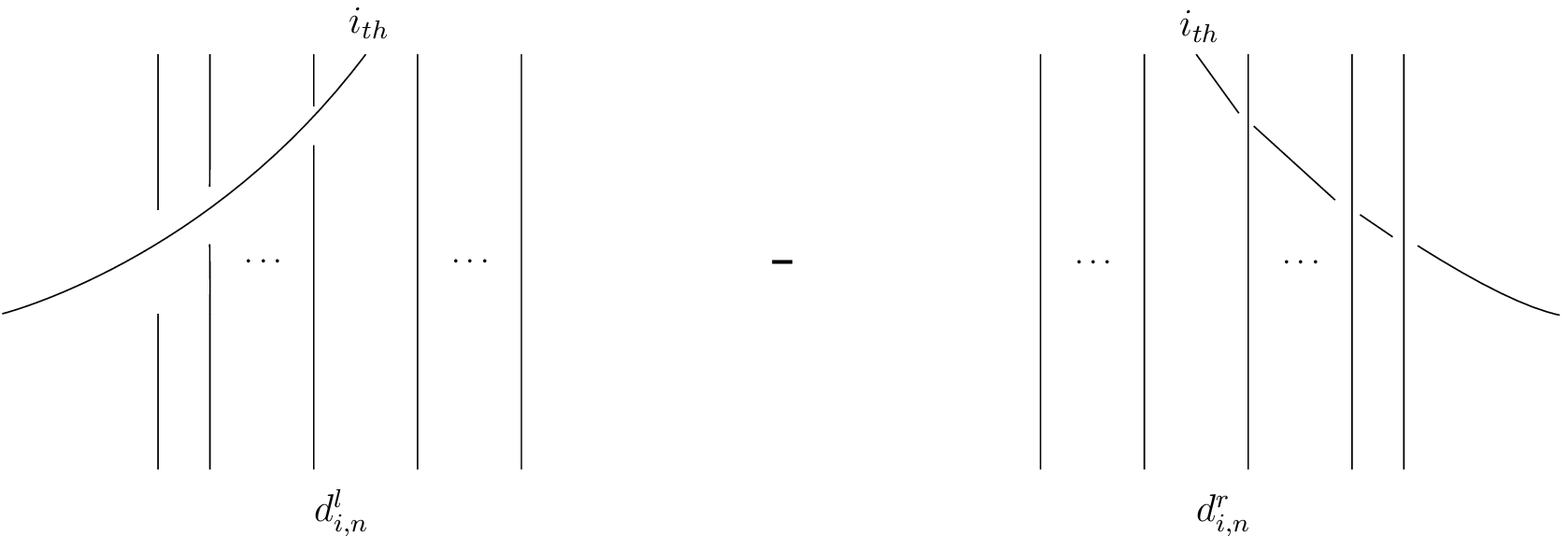} }
\centerline{\footnotesize{Figure. 2.1; a face map $d_{i,n}$}}\ \\
The meaning of $d^{l}_{i,n},$ is illustrated in Figure 2.2; $d^{r}_{i,n}$ can be described similarly.\ \\

\centerline{\includegraphics[scale=0.58]{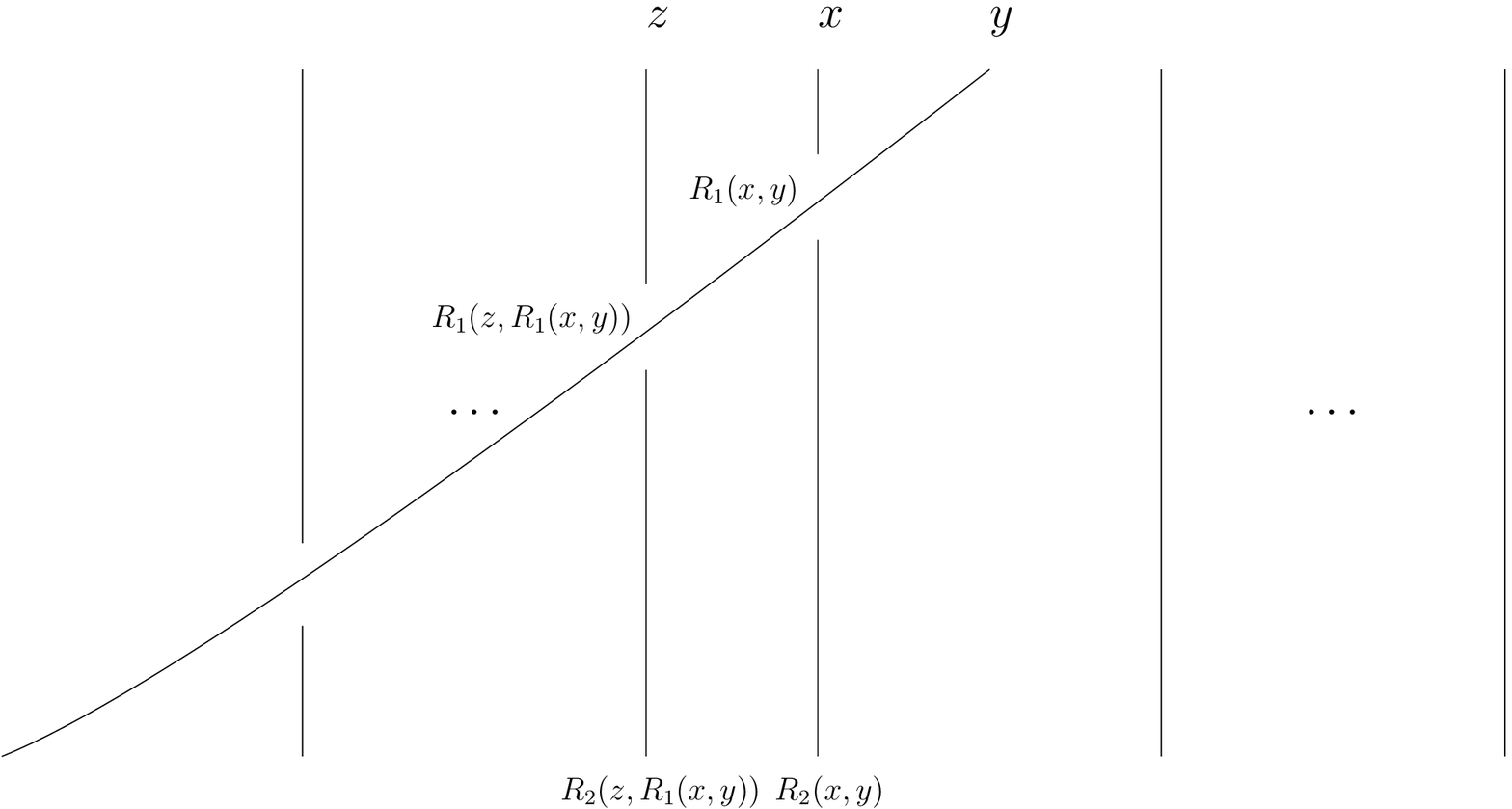} }
\centerline{\footnotesize{Figure. 2.2; a face map $d_{i,n}^{l}$}}\ \\

We have an $n-$tuple as an input and each strand carries the corresponding element of the $n-$tuple.  We track down the graph from top to bottom, and at each crossing we apply the fixed Yang-Baxter operator with input the ordered pair consists of two elements carried by the two strands right above the crossing(as in Figure 2.3).\\ \ 

\centerline{\includegraphics[scale=0.8]{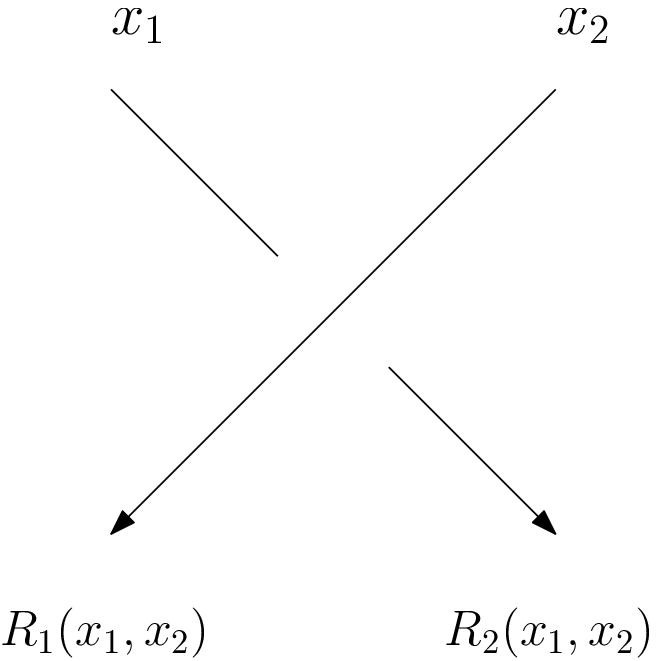} }
\centerline{\footnotesize{Figure. 2.3; encodding $R$ at each crossing}}
 Then the left strand after the crossing carries the $R_{1}$ function value and the right strand after the crossing carries the $R_{2}$ function value.  In the end, we ignore the element carried by the left most strand, and this procedure generates an $(n-1)-$tuple consisting of the $n-1$ elements carried by the other $n-1$ strands at the bottom.

One can easily check $(X^{n},d_{i,n}^{\epsilon})$ form a pre-cubic set, which implies that $(C^{YB}_{*}(X),\partial^{G}_{n})$ is a chain complex.  We define $H^{G}_{n}=ker\partial^{G}_{n}/im\partial^{G}_{n+1}$ as the ``graphic" version of Yang-Baxter homology group.

\section{Equivalence of two homology theories}

\begin{theorem}
The ``algebraic" and ``graphic" definitions of Yang-Baxter homology coincide.  More generally, chain complexes leading to above homologies are isomorphic.
\end{theorem}

\begin{proof}

Consider an $n-$dimensional cube $I^{n}$ and denote it as $I_{1}\times \cdots \times I_{n}.$ For any coloring, say $L(I_{1}\times \cdots \times I_{n})=(x_{1}, \dots, x_{n}),$ and an $(n-1)-$face $\mathcal{J}=I_{1}\times \cdots \times I_{i-1} \times 1_{i} \times I_{i+1} \times \cdots \times I_{n},$ we are going to demonstrate that $L(\mathcal{J})=d^{l}_{i,n}(x_{1}, \dots, x_{n}).$ To see this, we need to calculate the coloring of the initial path $(a_{1}, \dots, a_{n-1})$ of this $(n-1)-$face. By definition,\ \\

$a_{1}=  I_{1}\times 0_{2} \times 0_{3} \times \cdots \times 0_{i-1} \times 1_{i} \times 0_{i+1} \times \cdots \times 0_{n}$

$a_{2}=  1_{1}\times I_{2} \times 0_{3} \times \cdots \times 0_{i-1} \times 1_{i} \times 0_{i+1} \times \cdots \times 0_{n}$

$\vdots$

$a_{i-1}=1_{1}\times 1_{2} \times 1_{3} \times \cdots \times I_{i-1} \times 1_{i} \times 0_{i+1} \times \cdots \times 0_{n}$

$a_{i}=  1_{1}\times 1_{2} \times 1_{3} \times \cdots \times 1_{i-1} \times 1_{i} \times I_{i+1} \times 0_{i+2} \times \cdots \times 0_{n}$

$a_{i+1}=1_{1}\times 1_{2} \times 1_{3} \times \cdots \times 1_{i-1} \times 1_{i} \times 1_{i+1} \times I_{i+2} \times \cdots \times 0_{n}$

$\vdots$

$a_{n-2}=1_{1}\times 1_{2} \times 1_{3} \times \cdots \times 1_{i-1} \times 1_{i} \times 1_{i+1} \times \cdots \times I_{n-1} \times 0_{n}$

$a_{n-1}=1_{1}\times 1_{2} \times 1_{3} \times \cdots \times 1_{i-1} \times 1_{i} \times 1_{i+1} \times \cdots \times 1_{n-1} \times I_{n}$\ \\

We need another sequence $(b_{2}, \dots, b_{i}),$ where \ \\

$b_{2}=1_{1}\times 0_{2} \times 0_{3} \times \cdots \times 0_{i-2} \times 0_{i-1} \times I_{i} \times 0_{i+1} \times \cdots \times 0_{n}$

$b_{3}=1_{1}\times 1_{2} \times 0_{3} \times \cdots \times 0_{i-2} \times 0_{i-1} \times I_{i} \times 0_{i+1} \times \cdots \times 0_{n}$

$\vdots$

$b_{i-1}=1_{1}\times 1_{2} \times 1_{3} \times \cdots \times 1_{i-2} \times 0_{i-1} \times I_{i} \times 0_{i+1} \times \cdots \times 0_{n}$

$b_{i}=1_{1}\times 1_{2} \times 1_{3} \times \cdots \times 1_{i-2} \times 1_{i-1} \times I_{i} \times 0_{i+1} \times \cdots \times 0_{n}$\ \\

For example, for $j=i,$ the edges of the square are\ \\

$e_{i-1}=1_{1}\times 1_{2} \times 1_{3} \times \cdots \times I_{i-1} \times 0_{i} \times 0_{i+1} \times \cdots \times 0_{n}$

$b_{i}=e_{i}=1_{1}\times 1_{2} \times 1_{3} \times \cdots \times 1_{i-2} \times 1_{i-1} \times I_{i} \times 0_{i+1} \times \cdots \times 0_{n}$

$b_{i-1}=1_{1}\times 1_{2} \times 1_{3} \times \cdots \times 1_{i-2} \times 0_{i-1} \times I_{i} \times 0_{i+1} \times \cdots \times 0_{n}$

$a_{i-1}=1_{1}\times 1_{2} \times 1_{3} \times \cdots \times I_{i-1} \times 1_{i} \times 0_{i+1} \times \cdots \times 0_{n}$\ \\

Since \ \\

$L(e_{i-1})=x_{i-1},$ $L(b_{i})=x_{i}$\ \\

we have\ \\

$L(a_{i-1})=R_{2}(L(e_{i-1}),L(b_{i}))=R_{2}(x_{i-1},x_{i})$

$L(b_{i-1})=R_{1}(L(e_{i-1}),L(b_{i}))=R_{1}(x_{i-1},x_{i})$\ \\

Once we know the color of $b_{j},$ we know the colors of $b_{j-1},$ and $a_{j-1}$ (at each iteration we deal with a square in Figure 3.1).  Thus, recursively, we know all the colors of $a_{j}s'$ i.e. we get an $(n-1)-$tuple. \ \\

\centerline{\includegraphics[scale=0.5]{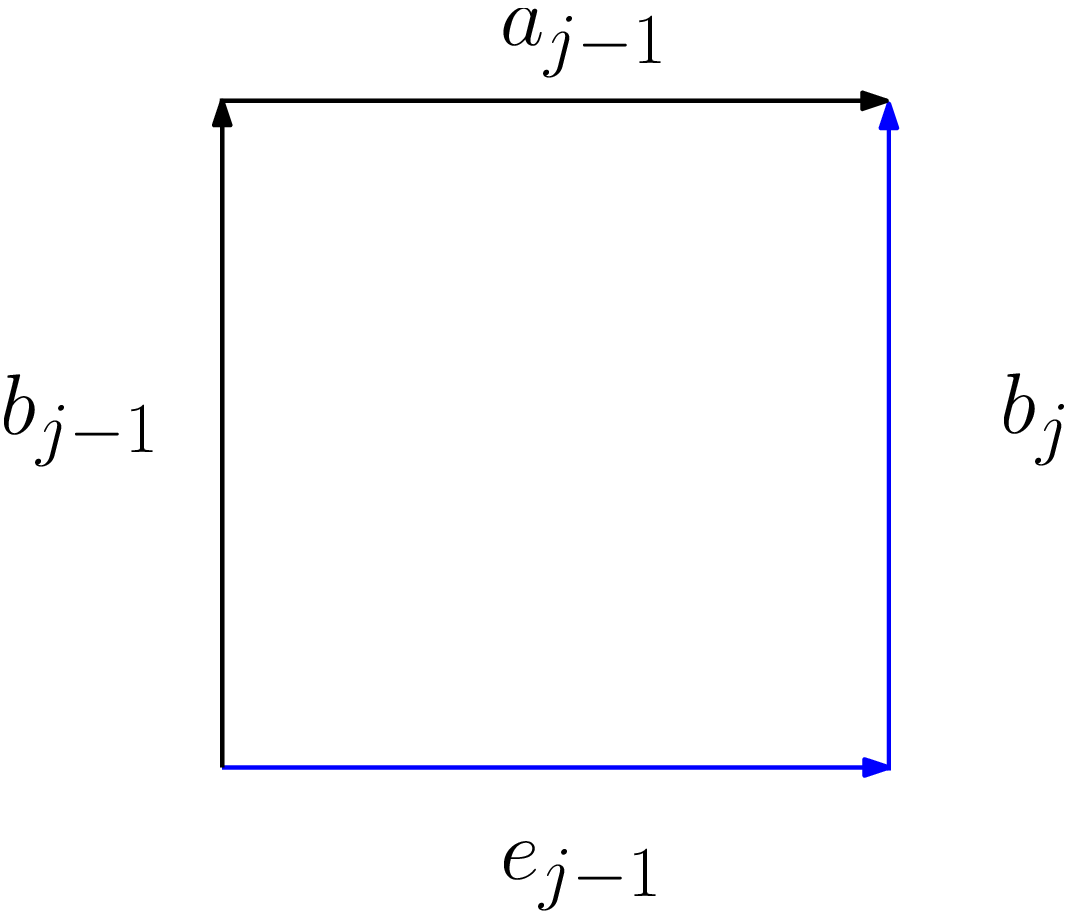} }
\centerline{\footnotesize{Figure. 3.1}}\ \\

In general, we have

$$ L(a_{j})=\left\{
\begin{array}{rcl}
R_{2}(L(e_{j}),L(b_{j+1}))       &      & {1\leq j \leq i-1}\\
x_{j+1}    &      & {i \leq j \leq n-1}

\end{array} \right. $$

$$ L(b_{j})=\left\{
\begin{array}{rcl}
R_{1}(L(e_{j}),L(b_{j+1}))       &      & {2\leq j \leq i-1}\\
x_{j}    &      & {j=i}

\end{array} \right. $$

We can see that this $(n-1)-$tuple is the same as the one given by $d^{l}_{i,n}$(compare with Figure 2.2).

Similarly, if we consider $\mathcal{J}=I_{1}\times \cdots \times I_{i-1} \times 0_{i} \times I_{i+1} \times \cdots \times I_{n},$ we can show $L(\mathcal{J})=d^{r}_{i,n}(x_{1}, \dots, x_{n})$.

As for the sign, we can directly calculate that, for $L(I_{1}\times \cdots \times I_{i-1} \times \epsilon_{i} \times I_{i+1} \times \cdots \times I_{n}),$ the sign is $(-1)^{n-i}(1-2\epsilon_{i}).$

Thus, the boundary map of ``algebraic" version and the boundary map of ``graphic" version only differ by a global sign, therefore, the considered chain complexes are isomorphic and they give isomorphic homology groups.
\end{proof}

\begin{example}\label{Example 3.1}

Comparison of the face maps corresponding to the $2-$face $I_{1}\times I_{2}\times 1_{3}$ of a cube and $d^{l}_{3,3}$.

\begin{minipage}[c]{8cm}
\includegraphics[scale=0.5]{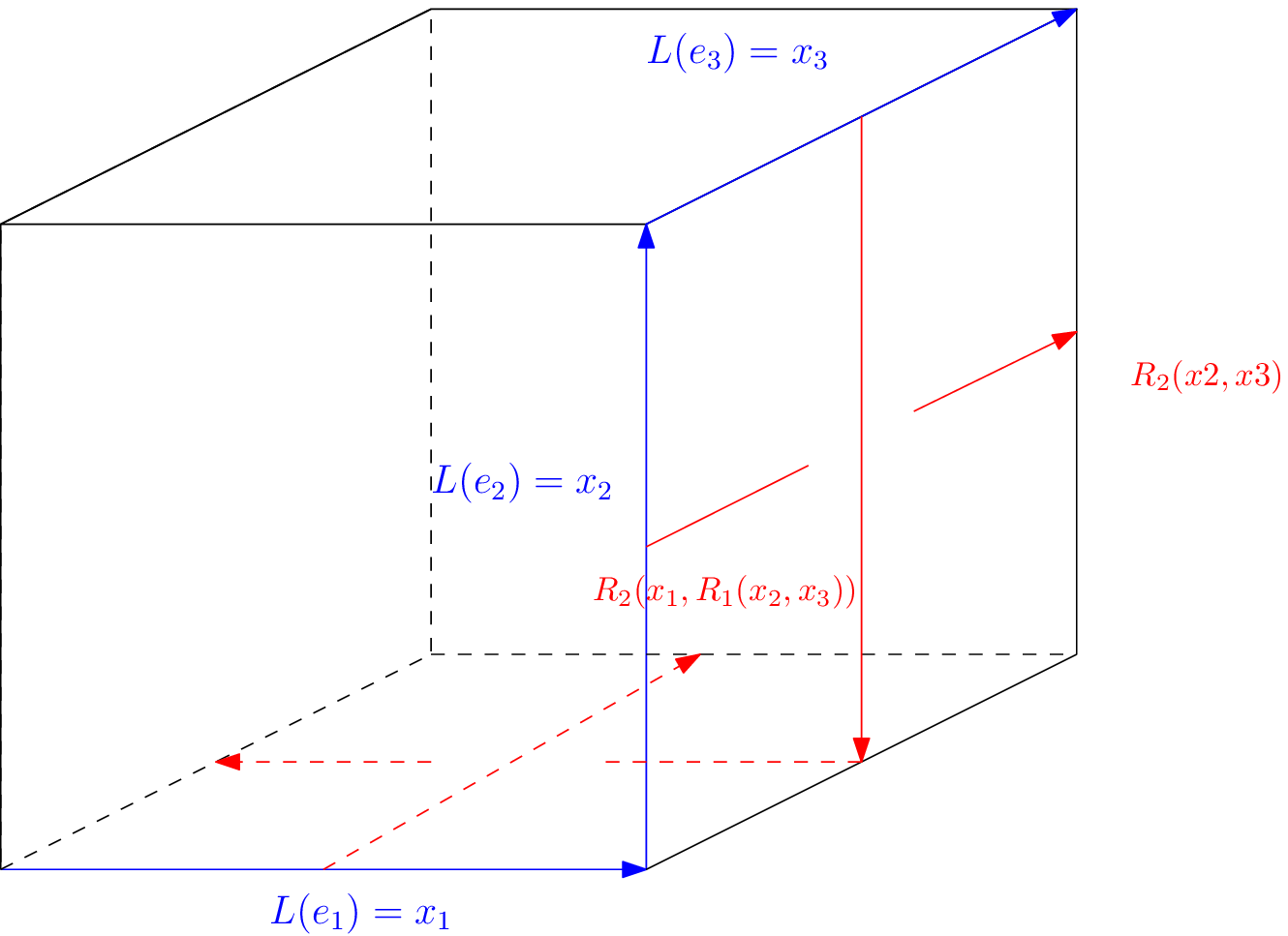} 
 \\
\footnotesize{Figure. 3.2; coloring of face $I\times I\times 1$}
\end{minipage} 
\begin{minipage}[c]{5cm}
\includegraphics[scale=0.5]{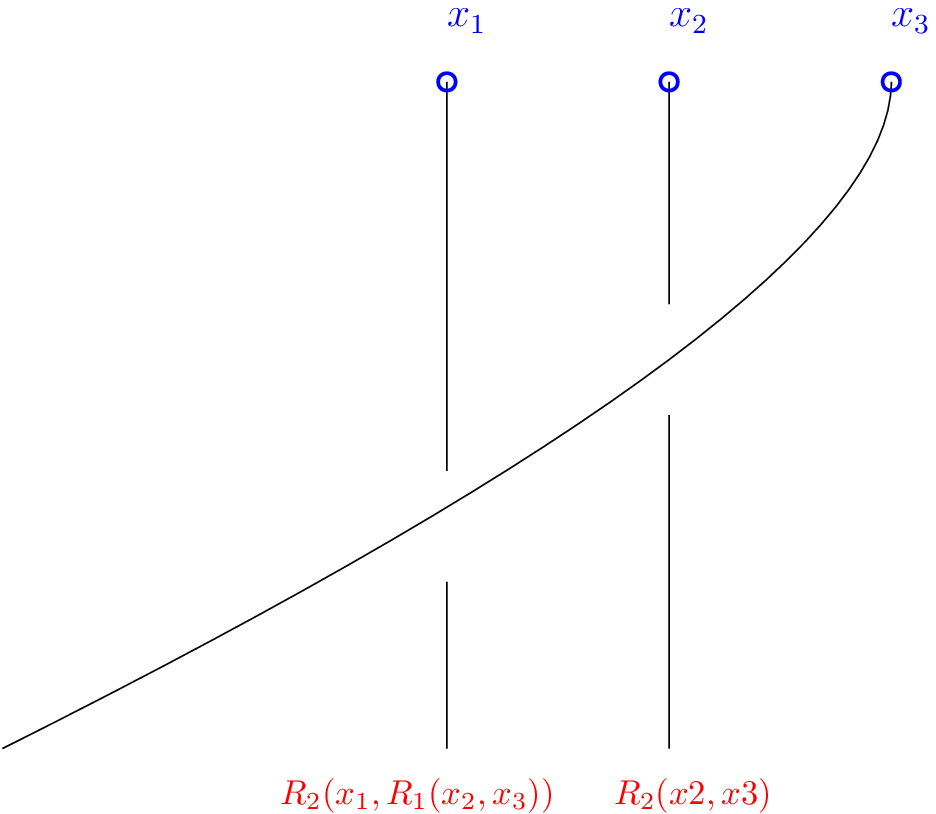} 
 \\
\footnotesize{Figure. 3.3; face map $d^{l}_{3,3}$}
\end{minipage} \ \\

From the figure above, we can see a ``curtain-like" object similar to the figure on the right hand side ``climbing" on the faces of the cube.  For higher dimensions, the similar consideration holds.  Therefore this also gives a way to visualize the equivalence.

\end{example}

\section{Homology for unital Yang-Baxter operators}

This ``graphic" definition of Yang-Baxter homology was motivated by the homology theory of self-distributive systems \cite{Prz-1,Prz-2}, for example shelves, racks, and quandles.  Similarly, we can define one-term and two-term homology not only for set-theoretic Yang-Baxter operators but also for pre-Yang-Baxter operators.  Furthermore, we can extend the definition of pre-Yang-Baxter homology to pre-Yang-Baxter operators with Yang-Baxter wall (see Definition 4.2). We will give the general definitions below and discuss some properties of these homology theories.

\subsection{One-term Yang-Baxter homology}

Let $k$ be a commutative ring with identity, $V=kX$ be a free $k$ module with basis $X$ and $M$ be a right $k-$module.  We define in Definition 4.2 the one-term pre-Yang-Baxter chain complex $\mathcal{C}^{YB}=(C_{n},M,\partial_{n})$ from the pre-simplicial module $(C_{n},M,d_{i})$. 

First we recall the notion of a pre-simplicial module.
\begin{definition}\label{Definition 4.1}
The pre-simplicial module ${\mathcal M}$ is a collection of modules $M_n$, $n\geq 0$, together with maps,
called  face maps or face operators,
$$d_i: M_n \to  M_{n-1}, \ \ 0\leq i \leq n, $$
such that:
$$d_id_j = d_{j-1}d_i,\ \ 0\leq i < j \leq n, $$
we define a chain complex with chain modules $M_n$ and a boundary map $\partial_{n}: M_n \to M_{n-1}$ given by:
$$\partial_{n}=\sum_{i=0}^n (-1)^i d_i$$
\end{definition}
We are ready to define the pre-Yang-Baxter pre-simplicial module.
\begin{definition}\label{Definition 4.2}(\cite{Leb,Prz-2})
Consider a linear map $R^{W} :M\otimes V \to M,$ such that $R^{W}\circ (R^{W}\otimes id_{V})=R^{W}\circ (R^{W}\otimes id_{V})\circ (id_{M}\otimes R)$ as shown graphically in Figure 4.1, we call this the left wall condition.\ \\

\centerline{\includegraphics[scale=0.7]{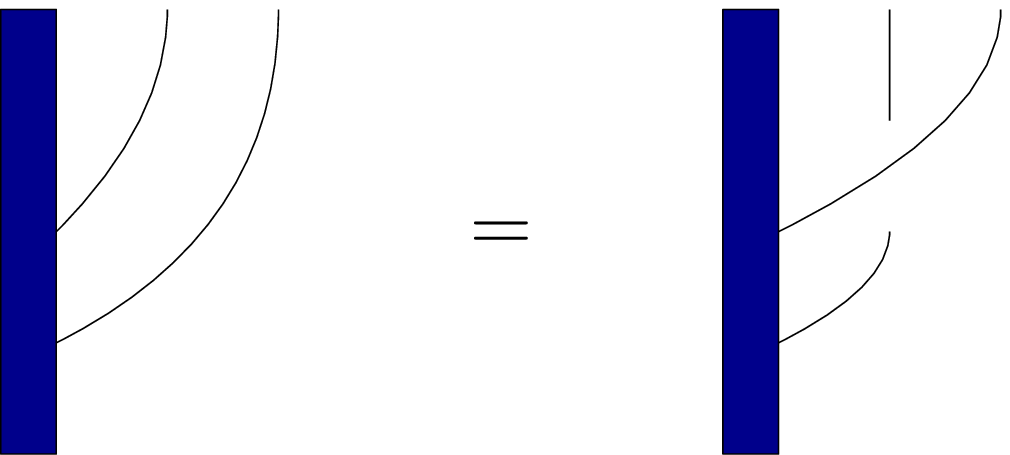} }
\centerline{\footnotesize{Figure. 4.1; the left wall condition}}\ \\

Let $C_{n}=M\otimes V^{\otimes n}$ and the face map $d_{i}=d_{i,n}:C_{n}\to C_{n-1}$ is defined by 
$$d_{i}=$$ $$(R^{W}\otimes id^{\otimes n-1})\circ (id_{M}\otimes R \otimes id^{\otimes n-2})\circ (id_{M}\otimes id\otimes R\otimes id^{\otimes n-3})\circ ... \circ (id_{M}\otimes id^{\otimes i-2}\otimes R\otimes id^{n-i})$$
 We can interpret the face maps through Figure 4.2


\centerline{ \includegraphics[scale=0.7]{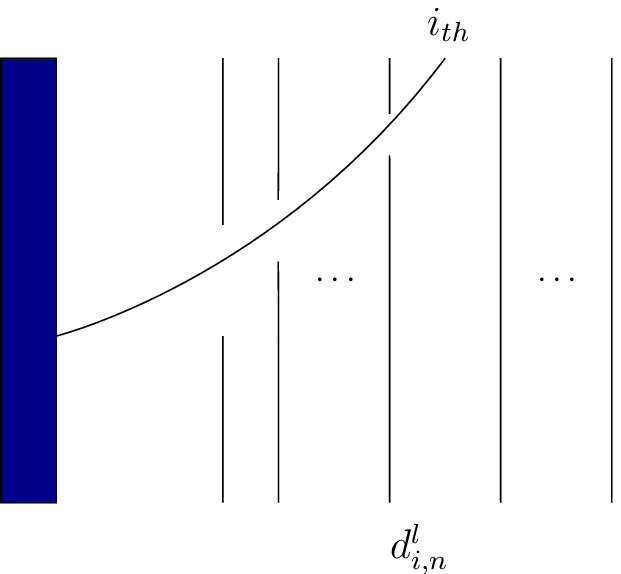}  } 
\centerline{\footnotesize{Figure. 4.2; face map $d_{i,n}$}} 

\ \\
$(C_{n},M,d_{i})$ is a pre-simplicial module, and $(C_{n},M,\partial_{n})$ is the one-term pre-Yang-Baxter chain complex.  It's homology is called the one-term pre-Yang-Baxter homology $H_{n}(R,R^{W}).$
\end{definition}

\begin{proposition}\label{Proposition 4.3}
Let $(C_{n},M,\partial_{n})$ be the chain complex of the one-term pre-Yang-Baxter homology.  For a fixed element $v\in V,$ consider the map $f_{n}:C_{n} \to C_{n}$, defined by $f_{n}(a)=d_{n+1,n+1}(a \otimes v)$, where $a \in C_{n},$ then this is a chain map. We have 
\begin{enumerate}
\item[(1)]$(f_{n})_{*}(H_{n})=0$.
\item[(2)]If there is an element $v \in V$ such that $f_{n}$ is invertible,  then the one-term pre-Yang-Baxter homology is trivial.
\end{enumerate}

\end{proposition}

\begin{proof}
We construct a chain homotopy $P_{n}$ between $(-1)^{n+1}f_{n}$ and zero map, in particular showing that $f_{n}$ is a chain map.  The chain homotopy $P_{n}=C_{n} \to C_{n+1}$ is defined by $P_{n}(v_{i_{1}}\otimes ...\otimes v_{i_{n}})=(-1)^{n}v_{i_{1}}\otimes ...\otimes v_{i_{n}}\otimes v.$ We check: $\partial_{n+1}P_{n}+P_{n-1}\partial_{n}=\sum_{i=1}^{n+1}(-1)^{i}d_{i,n+1}P_{n}+\sum_{j=1}^{n}(-1)^{j}P_{n}d_{i,n}=(-1)^{n+1}d_{n+1,n+1}P_{n}=(-1)^{n+1}f_{n}.$ $(1)$ follows because $\{f_{n}\}$ is chain homotopic to the $0$ map, therefore, $(f_{n})_{*}$ is the $0$ map on homology.  $(2)$ follows since if $f_{n}$ is invertible, so is $(-1)^{n+1}f_{n},$ thus  $H_{n}(C_{*})$ is isomorphic to $H_{n}(C_{*})$ through zero map.  This shows that $H_{n}(C_{*})={0}.$

\end{proof}

In the case that $M=k,$ and of $V$ acting on $k$ trivially, the condition making $(V^{\otimes n},d_{i,n})$ a pre-simplicial module is equivalent to that the sum of each column of the $R$ matrix is $1$, which we call the column unital condition (e.g. stochastic matrices satisfy the condition).

\begin{corollary}
Let $M=k,$ $V$ act on $k$ trivially, and $R$ be a set-theoretic Yang-Baxter operator. If for any pair $(B,D),$ there is a unique $A$ such that $R(A,B)=(R_{1}(A,B),R_{2}(A,B))=(C,D),$ then conditions in Proposition 4.3 hold.
In particular, biracks satisfy this condition (see Definition 3.1 condition 3, that is right invertibility, in \cite{CES}).
\end{corollary}

\begin{proof}
We need to show that for any $n-$tuple $(y_{1},...,y_{n})$ in $C_{n},$ there exists unique $n-$tuple $(x_{1},...,x_{n})$ in $C_{n}$ such that $f_{n}((x_{1},...,x_{n}))=d_{n+1,n+1}((x_{1},...,x_{n},v)=(y_{1},...,y_{n}).$  Since we know $y_{n}$ and $v$, we get the values of $x_{n}$ and $R_{1}(x_{n},v)$ uniquely.  Once we have $R_{1}(x_{n},v),$ together the value of $y_{n-1},$ we get the value of $x_{n-1}$ and $R_{1}(x_{n-1},R_{1}(x_{n},v))$ uniquely.  Thus by this iteration, we get the $n-$tuple $(x_{1},...,x_{n})$ uniquely and this shows $f_{n}$ is invertible.
\end{proof}

\begin{example}\label{Example 4.2.1}
(Compare \cite{CES}) Let $F$ be a commutatinve ring with identity.  Let $k=F[s^{\pm 1},t^{\pm 1}]/(1-s)(1-t),$ then $$R(x,y)=(R_{1}(x,y),R_{2}(x,y))=((1-s)x+ty,sx+(1-t)y)$$ is a set-theoretic Yang-Baxter operator satisfying the conditions in Corollary 4.4.  This holds because for any given $y$ and $a=R_{2}(x,y)=sx+(1-t)y$, we can solve $x=s^{-1}(a-(1-t)y).$  Thus the one-term homology of this operator is trivial.
\end{example}

\begin{remark}
Pre-Yang-Baxter coming from racks $(X,*)$ where $R(a,b)=(b,a*b)$ satisfies the conditions in Proposition 4.3.  Thus it has zero one-term homology(see \cite{Prz-1}).
\end{remark}
\subsection{Two-term Yang-Baxter homology} 

Let $k$ be a commutative ring with identity, $V=kX$ be a free $k$ module with basis $X$, $M$ be a right $k-$module and $N$ be a left $k-$module.  We define  in Definition 4.8 the two-term pre-Yang-Baxter chain complex $\mathcal{C}^{YB}=(C_{n},M,N,\partial_{n})$ from the pre-cubical module $(C_{n},M,N,d_{i}^{\epsilon})$. 

\begin{definition}\label{Definition 4.1}
The pre-cubical module ${\mathcal M}$ is a collection of modules $M_n$, $n\geq 0$, together with maps,
called  face maps or face operators,
$$d_i^\epsilon : M_n \to  M_{n-1}, \ \ 1\leq i \leq n, \epsilon=0,1$$
such that:
$$d_i^\epsilon d_j^\delta = d_{j-1}^\delta d_i^\epsilon,\ \ 1\leq i < j \leq n, \epsilon, \delta=0,1$$
we define a chain complex with chain groups $M_n$ and a boundary map $\partial_{n}: M_n \to M_{n-1}$ given by:
$$\partial_{n}=\sum_{i=1}^n (-1)^i (d_i^0 - d_i^1).$$
\end{definition}
We are ready to define the pre-Yang-Baxter pre-cubical module.
\begin{definition}\label{Definition 4.2}(\cite{Leb,Prz-2})
Consider a linear map $R^{W}_{l} :M\otimes V \to M,$ such that $R^{W}_{l}\circ (R^{W}_{l}\otimes id_{V})=R^{W}_{l}\circ (R^{W}_{l}\otimes id_{V})\circ (id_{M}\otimes R),$ and a linear map $R^{W}_{r} :V\otimes N \to N,$ such that $R^{W}_{r}\circ (id_{V}\otimes R^{W}_{r})=R^{W}_{r}\circ (id_{V}\otimes R^{W}_{r})\circ (R\otimes id_{N})$ we call them the left and right wall conditions(graphically see Figure 4.3 and Figure 4.4).\ \\

\centerline{\includegraphics[scale=0.7]{lcondition.eps} }
\centerline{\footnotesize{Figure. 4.3; left wall condition}}\ \\

\centerline{\includegraphics[scale=0.7]{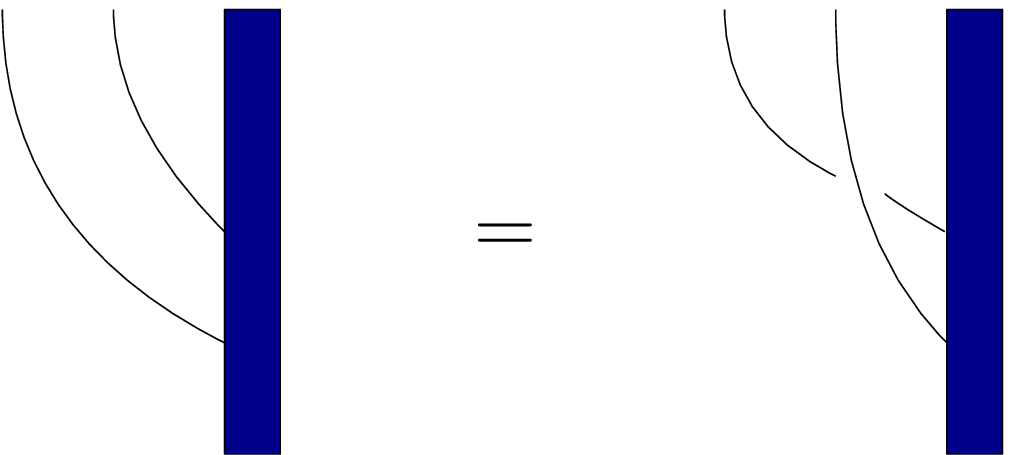}  }
\centerline{\footnotesize{Figure. 4.4; right wall condition}}\ \\

Let $C_{n}=M\otimes V^{\otimes n}\otimes N$ and face maps $d_{i}^{\epsilon}: C_{n}\to C_{n-1}$ are given by 
$$d_{i}^{l}=(R^{W}_{l}\otimes id^{\otimes n-2}\otimes id_{N})\circ (id_{M}\otimes R \otimes id^{\otimes n-1}\otimes id_{N})\circ (id_{M}\otimes id\otimes R\otimes id^{\otimes n-3}\otimes id_{N})\circ ... $$
$$...\circ (id_{M}\otimes id^{\otimes i-2}\otimes R\otimes id^{\otimes n-i}\otimes id_{N})$$ 
and
$$d_{i}^{r}=(id_{M}\otimes id^{\otimes n-1}\otimes R^{W}_{r})\circ (id_{M}\otimes id^{\otimes n-2}\otimes R\otimes id_{N})\circ (id_{M}\otimes id^{\otimes n-3}\otimes R\otimes id\otimes id_{N})\circ...$$
$$...\circ (id_{M}\otimes id^{\otimes i-1} \otimes R\otimes id^{\otimes n-i-1}\otimes id_{N})$$

 We can interpret the face maps $d_{i}^{l},$ $d_{i}^{r}$ and their difference through Figure 4.5.\ \\

 \centerline{\includegraphics[scale=0.7]{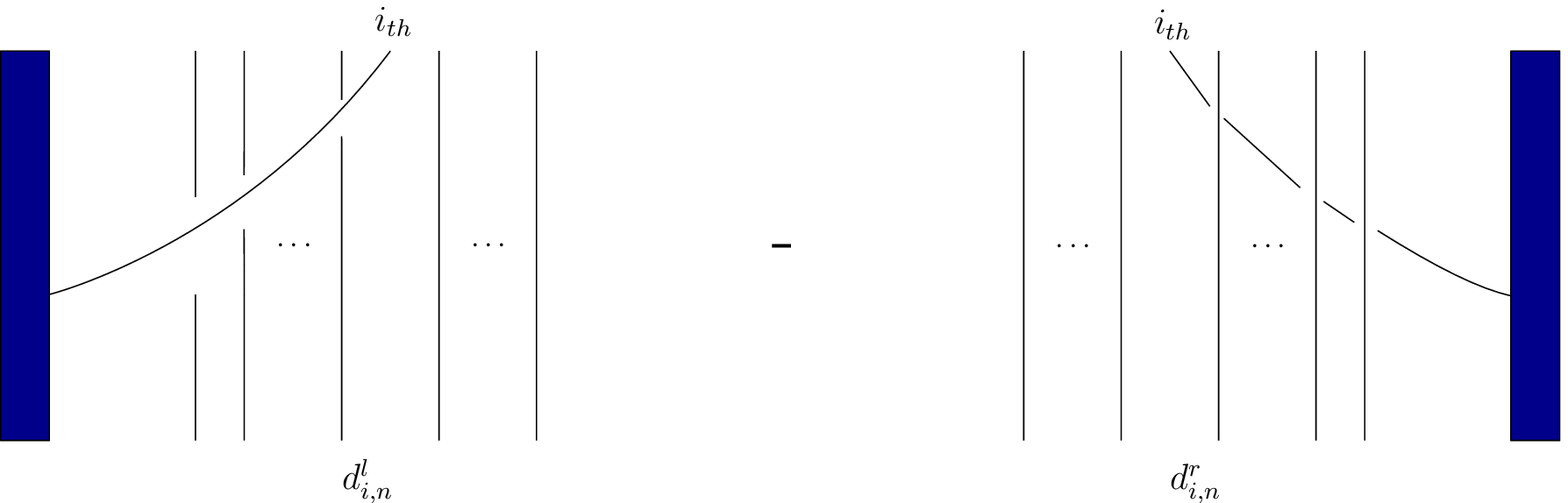} }
\centerline{\footnotesize{Figure. 4.5; face map $d_{i,n+1}$}}\ \\

$(C_{n},M,N,d_{i}^{\epsilon})$ is a pre-cubical module, and $(C_{n},M,N,\partial_{n})$ is the two-term pre-Yang-Baxter chain complex.
\end{definition} 

In the case that $M=k=N,$ and the action of $V$ on $k$ is the trivial action, the conditions making $(V^{\otimes n},d_{i,n}^{\epsilon})$ a pre-cubical module is equivalent to saying that $R$ has the column unital condition.

\begin{example}\label{Example 4.2.1}
We give a family of unital Yang-Baxter operator $R:V\otimes V \to V\otimes V,$ where $V=k\left\{v_{1},...,v_{m}\right\},$ $k=\mathbb{Q}[y,y^{-1}]$ and $m$ is a positive integer.  For any given $m$, $R$ can be represented by its coefficients,
$$R_{i j}^{k l}=\left\{
                  \begin{array}{ll}
                    1, & \hbox{if i=j=k=l;} \\
                    1, & \hbox{if l=i$>$j=k;} \\
                    y^2, & \hbox{if l=i$<$j=k;} \\
                    1-y^2, & \hbox{if k=i$<$j=l;} \\
                    0, & \hbox{otherwise.}
                  \end{array}
                \right.$$
                
These family of Yang-Baxter operators are unital, for example when $m=2$, it is 
$$
 \left[
 \begin{matrix}
   1 & 0 & 0 & 0 \\
   0 & 1-y^{2} & 1 & 0 \\
   0 & y^{2} & 0 & 0 \\
   0 & 0 & 0 & 1
  \end{matrix}
  \right]
$$
\end{example}

Computation shows interesting pattern in the two-term Yang-Baxter homology of this family of Yang-Baxter operators.  

\begin{conjecture}(\cite{Wan})
When $m=2,$ $H_{n}=k^{2}\bigoplus (k/(1-y^{2}))^{a_{n}}\bigoplus (k/(1-y^{4}))^{s_{n-2}},$ where $s_{n}=\Sigma_{i=1}^{n+1}f_{i}$ is the partial sum of Fibonacci sequence, where $f_{1}=1=f_{2}$ and $a_{n}$ is given by $2^{n}=2+a_{n-1}+s_{n-2}+a_{n}+s_{n-1}$ with $a_{1}=0.$  We verified the conjecture for $n\leq 10.$
\end{conjecture}
 What fascinates us is that this family of Yang-Baxter operators come from the Yang-Baxter operators giving $sl_{m}$ polynomial invariants of links (substitutions to the Homflypt polynomial).  For example, when $m=2$, the matrix is

$$
 \left[
 \begin{matrix}
   -q & 0 & 0 & 0 \\
   0 & q^{-1}-q & 1 & 0 \\
   0 & 1 & 0 & 0 \\
   0 & 0 & 0 & -q
  \end{matrix}
  \right]
$$

If we divide elements in each column by the sum of those elements and make the substitution $y=(1+q^{-1}-q)^{-1/2},$ we will get the matrix in Example 4.9.  In general, we can get our family in Example 4.9 in a similar way (normalizing columns) and again they are Yang-Baxter operators.  More interestingly, this new family of Yang-Baxter operators also provide $sl_{m}$ polynomial invariants of links \cite{Wan}. This fact is implicit in \cite{Jon-2}.

 \subsection{Cocycle invariant from Yang-Baxter homology} 
 
 In the paper \cite{CES} where Carter, Elhamdadi and Saito define their set-theoretic Yang-Baxter homology, they also define a 2-cocycle link invariant from the set-theoretic Yang-Baxter homology.  This can also be done in the case of column unital Yang-Baxter operators.  In \cite{Prz-3}, it demonstrates the third Reidemeister move preserve the (co)homology of the (co)cycle constructed form a knot diagram. See Figure 4.6 and \cite{Prz-3} for details.

\centerline{\includegraphics[scale=0.7]{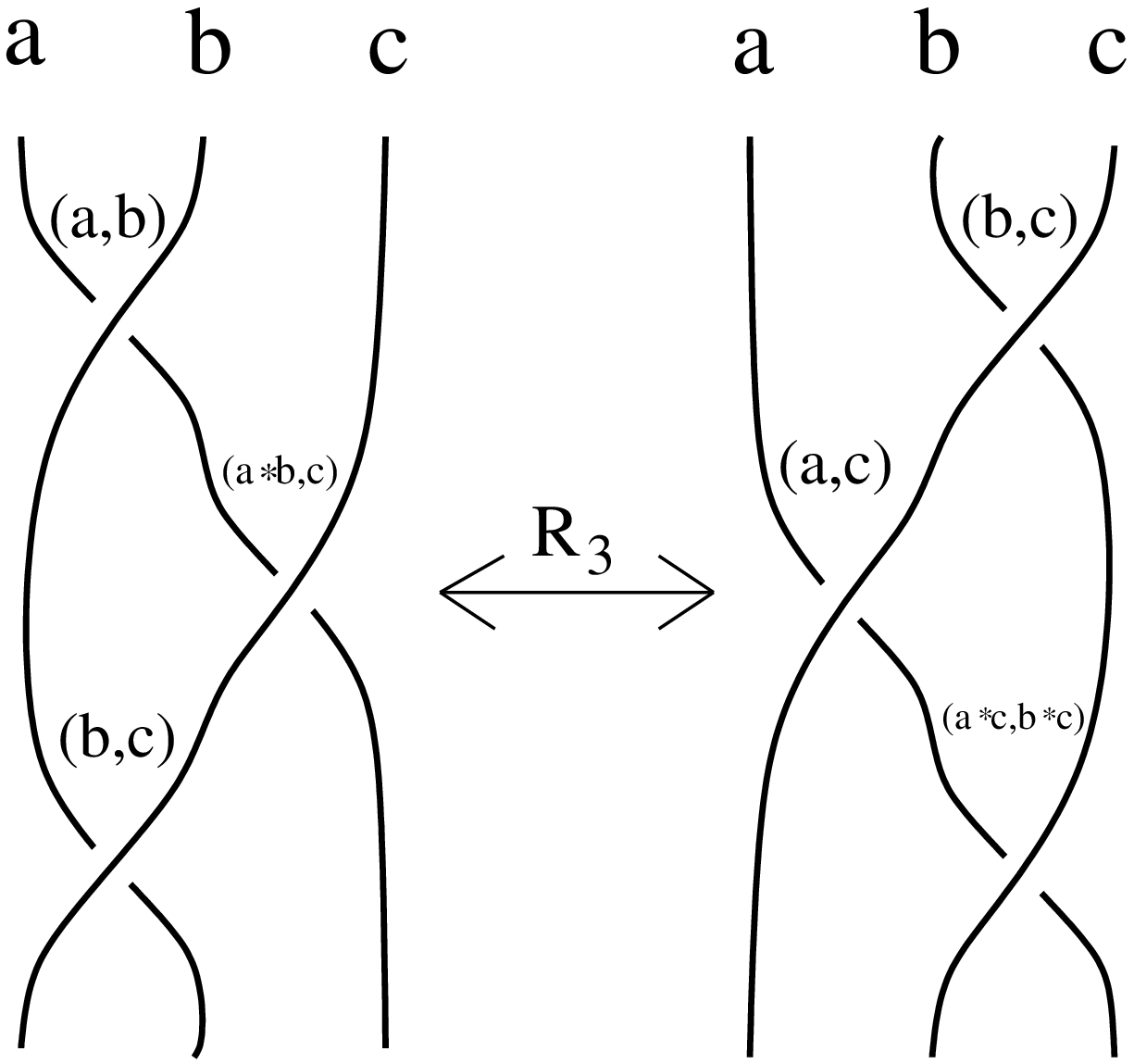}} 
\centerline{\footnotesize{Figure. 4.6; relation between third Reidemeister move and 2-(co)cycle condition}}


\begin{thebibliography}{pape-1}
\bibitem[Bax]{Bax} R.~J.~Baxter, Partition function of the Eight-Vertex lattice model
{\it Annals of Physics}, 70, 1, 1972, 193-228
\bibitem[CES]{CES} J.~S.~Carter,M.~Elhamdadi,M.~Saito, Homology theory fpr the set-theoretic Yang-Baxter equation and knot invariants from generalizations of quandles,
{\it Fund.~Math}, 184(2004),31-54
\bibitem[Jon-1]{Jon-1} V.~F.~R. Jones, On knot invariants related to some statistical mechanical
models, {\it Pacific J. Math.}, 137 (1989), no. 2, 311-334.
\bibitem[Jon-2]{Jon-2} V.~F.~R.~ Jones, Hecke Algebra Representations of Braid Groups and Link Polynomials, {\it Annals of Mathematics} Vol. 126, No. 2 (Sep., 1987), pp. 335-388
\bibitem[Leb]{Leb} V.~Lebed, Homologies of algebraic structures via braidings and quantum shuffles
{\it J.~Algebra}, 391: 152-192(2013);\\
e-print:\ {\tt arXiv:1204.3312 [math.CT]}
\bibitem[Prz-1]{Prz-1} J.~H.~Przytycki,Distributivity versus associativity in the homology theory of algebraic structures,
 {\it Demonstratio Math.}, 44(4), 2011, 823-869; \\
 e-print:\ {\tt arXiv:1109.4850 [math.GT]}
\bibitem[Prz-2]{Prz-2} J.~H.~Przytycki, Knots and distributive homology: from arc colorings to Yang-Baxter homology,
 Chapter in: {\it New Ideas in Low Dimensional Topology}, World Scientific, Vol. 56, March-April 2015, 413-488.
e-print: \ {\tt arXiv:1409.7044 [math.GT]}
\bibitem[Prz-3]{Prz-3}J.~H.~Przytycki, Curtain homology:\ 2-(co)cycle invariants from Yang-Baxter operators,
in preparation.
\bibitem[Tur]{Tur} V.~G.~Turaev, The Yang-Baxter equation and invariants of links, {\it Invent.
Math.}, 92,1988, 527-553.
\bibitem[Wan]{Wan} X.~Wang, PhD Thesis George Washington University.
In progress.
\bibitem[Yan]{Yan} C.~N.~Yang, Some Exact Results for the Many-Body Problem in one Dimension with Repulsive Delta-Function Interaction
{\it Phys. Rev. Lett.} 1967, 1312
\end{thebibliography}
 \end{document}